\documentclass[12pt,draftcls,onecolumn]{IEEEtran}
\usepackage{mathrsfs,amsmath,latexsym,amssymb,epsfig}

\newtheorem{thm}{Theorem}[section]

\newcommand{\argmax}{\mathop{\mbox{\rm arg\,max}}}

\newcommand{\UCB}{\mathop{\mbox{\footnotesize UCB1}}}
\newcommand{\UCT}{\mathop{\mbox{\footnotesize UCT}}}
\begin{document}
\sloppy

\title{On the Convergence Rate of MCTS for the Optimal Value Estimation in Markov Decision Processes}
\author{Hyeong Soo Chang
\thanks{H.S. Chang is with the Department of Computer Science and Engineering 
at Sogang University, Seoul 121-742, Korea. (e-mail:hschang@sogang.ac.kr).}%
}

\maketitle
\begin{abstract}
A recent theoretical analysis of a Monte-Carlo tree search (MCTS) method properly modified from the ``upper confidence bound applied to trees" (UCT) algorithm established a surprising result, due to a great deal of empirical successes reported from heuristic usage
of UCT with relevant adjustments for various problem domains in the literature, that its rate of convergence of the expected absolute error to zero is $O(1/\sqrt{n})$ in estimating the optimal value at an initial state in a finite-horizon Markov decision process (MDP), where $n$ is the number of simulations.
We strengthen this dispiriting slow convergence result by arguing within a simpler algorithmic framework in the perspective of MDP, apart from the usual MCTS description, that the simpler strategy, called ``upper confidence bound 1" (UCB1) for multi-armed bandit problems, when employed as an instance of MCTS by setting UCB1's arm set to be the policy set of the underlying MDP, has an asymptotically faster convergence-rate of $O(\ln n / n)$.
We also point out that the UCT-based MCTS in general has the time and space complexities that depend on the size of the state space in the worst case, which contradicts the original design spirit of MCTS.
Unless heuristically used, UCT-based MCTS has yet to have theoretical supports for its applicabilities.
\end{abstract}


\section{Introduction}

Markov decision process (MDP), also known as stochastic dynamic programming, is a fundamental model 
for solving sequential decision making problems under uncertainty, developed by Bellman, 
(see, e.g.,~\cite{bert2}~\cite{changbook}~\cite{pow} and references therein) to generally 
maximize or minimize the expected total accumulated reward or cost over a finite or infinite horizon
with or without discounting.
The problem is to obtain the optimal value at each state and/or an optimal action to take
at each state over time, i.e., an optimal policy, which achieves the optimal value 
at each state if the policy is followed.

Policy iteration, value iteration, linear programming (for infinite horizon MDPs), and 
backward induction (for finite horizon MDPs) are the well-known \emph{exact} algorithms for the problem.
Unfortunately, MDP suffers from the \emph{the curse of dimensionality} in that the complexity of modelling 
and the complexity of the exact methods can grow exponentially in the size of the problem.
There exists a great body of the literature about computational (heuristic) methodologies for providing 
``approximate" solutions to MDPs while addressing the scalability issue 
(see, e,g.,~\cite{changbook},~\cite{pow},~\cite{bert2} and references therein).
Most approaches have focused on being \emph{off-line} in that we compute an (approximate) optimal policy \emph{in advance} and 
then apply the control law to the underlying system. This off-line approach still has potentially a high complexity 
in solving large MDPs.

Suppose that we have decision-making situation where the controller needs to acts over the trajectory 
of the visited states as the system evolves over time. In this situation, rather than computing the 
optimal value at every state in advance, we would better solve only the (finite-horizon) subproblem at hand of 
obtaining a near-optimal value (or a near-optimal action to take) \emph{at the current state only}.
In such \emph{on-line} setting (see, e.g,~\cite{changbook},~\cite{bert}), estimating the optimal 
value at the current state for a finite horizon is the crux of the solution process to obtain 
a near-optimal action but doing so needs to still face with the dimensionality issue.
In a breakthrough work by Chang \emph{et al.}~\cite{changams}, an algorithm called ``adaptive multi-stage sampling" (AMS)
was presented, which addresses this issue.
The basic idea is, when the action space is relatively small, to use a 
\emph{random sampling} for the next-state transition to approximate the expectation over the whole set of the 
reachable states by a sample average over the sampled next-states. 
The key
question is then how to select an action for sampling a next-state and to control the size 
of the number of the sampled next-states while guaranteeing a convergence to the optimal value 
when the number of samples is sufficiently large.

Specifically, given a sampled state (including the current root state), AMS selects $N$ possibly 
different actions over $N$ steps,
where this selection process adopts the idea of an exploration and exploitation process of
choosing the arms to be played~\cite{auer}~\cite{lai} for stochastic multi-armed bandit (MAB) problems
to minimize the criterion of the ``expected regret". An arm with the maximum index-value is chosen at each
step, where the index-value of an arm at a particular step is given with 
the sample average over the intermediate steps at which the arm had been played up to the step plus a term related
with upper confidence bound (UCB) of the sample average estimate.
In AMS, every time a particular action is selected, AMS samples a next-state from the given next-state
probability distribution associated with the action and the underlying sampled state at which the action was chosen.
The process of sampling a next-state (by selecting an action) is recursively done from the 
initial stage to the final stage in a manner of the depth-first search, starting from 
the current (root) state. AMS in a sense emulates backward induction over the sampled states.
When a recursive call made from a state is returned to the state, the estimate of the 
optimal value at the state is updated based on Bellman's optimality equation.
Due to the sampling process like the depth-first search, AMS follows a search path in $O((|A|N)^H)$-size
tree made of the sampled states if $A$ denotes the action set and $H$ is the horizon size.
Thus, the time-complexity of AMS is $O((|A|N)^H)$ but \emph{independent of the state space size} while
the estimate converges to the optimal value at the root state as $N$ increases and 
the convergence rate is $O(H \ln N/ N)$~\cite{changams}.

Inspired from AMS, Coulom studied a simulation-based approach in the context of ``planning" with the model of MDP and coined the name of the approach as ``Monte-Carlo tree search" (MCTS)~\cite{coulom}.
Kocsis and Szepesv\'{a}ri~\cite{kocsis} presented a more elaborated algorithm, called ``upper confidence bound applied to Trees" (UCT), as an instance of MCTS. It should be clarified that UCT was originally designed for solving finite-horizon MDPs (or solving approximately 
infinite-horizon discounted MDPs) with the goal of handling with MDPs that have large state spaces.
Since then, much attention has been paid to MCTS in various problem domains (see, e.g.,~\cite{browne}~\cite{shah}~\cite{bert}~\cite{macie} and the references therein) and a great deal of \emph{empirical} successes have been reported from heuristic usage of UCT with relevant adjustments in the literature and UCT has been widely considered as a standard algorithm when implementing MCTS.
The main characteristics 
of UCT is using a similar UCB method, as in AMS, but to ``build" or generate a next \emph{policy} to ``roll out" or to simulate from the past simulation results
(whereas AMS uses the UCB method to select an action that is used for sampling a next-state when going 
deeper in the depth-first tree search) and at the same time 
updating a tree made of the visited states by rolling out the policies generated for a bookkeeping process. 
The average of the accumulated 
sample-reward sums over a finite horizon, starting at the current root state, obtained by rolling out the policies generated by UCT 
is the estimate of the optimal value at the state for the horizon. It should be noted that Bellman's optimality principle is \emph{not} incorporated into the update process but the non-recursive 
process makes its sample complexity polynomially dependent on $H$ unlike the exponential dependence on $H$ in AMS.

Even with such a popularity of UCT-based MCTS, there has been no rigorous theoretical work about the convergence 
behavior of UCT or MCTS in general until Shah \emph{et al.}'s performance analysis~\cite{shah}
on a properly modified UCT. 
Throughout the note, we refer to the algorithm by Shah \emph{el al.} as UCT-C (UCT-corrected) because
they resolved (by correcting UCT) a critical issue regarding the convergence rate in achieving 
the asymptotic optimality (cf., Section~\ref{sec:UCT-C}).
Shah \emph{et al.} observed that the originally claimed result of $O(\ln n/n)$-rate of convergence to zero (or the upper bound $O(\ln n/n)$ on the absolute error) in~\cite{kocsis} is erroneous and established that UCT-C's rate of convergence to zero is $O(1/\sqrt{n})$ where $n$ is the number of the simulations, i.e., the number of the policies generated for simulations.
Since then, it is difficult to find any work or correspondence in the literature that treats about \emph{implication} of this result in the perspective of solution methodology of MDPs even if MCTS recently has been one of the highly appealing (control and optimization) topics (see, e.g.,~\cite{bert}).

We strengthen this dispiriting slow-convergence result by arguing within the simpler algorithmic framework, rather than the usual MCTS description (see, e.g.,~\cite{browne}), that 
the much simpler strategy for multi-armed bandit problems, called ``upper confidence bound 1" (UCB1) by Auer \emph{et al.}~\cite{auer}, when employed as an instance of MCTS by setting the arm set to be the policy set of the underlying MDP, has an asymptotically faster convergence-rate of $O(\ln n / n)$.
We also point out that MCTS in general has the time and space complexities that depend on the size of the state space, which contradicts the original design spirit of MCTS. It is supposed to overcome the curse of dimensionality problem, possibly at least no worse than AMS.
Unless heuristically used, UCT-based MCTS has yet to have theoretical supports for its applicabilities.

This note is organized as follows. In Section II, we describe the setting of finite-horizon MDPs,
and state the problem, and provides an algorithmic framework within which the algorithms considered in this note are explained as instances.
Section II describes UCB1 and its asymptotic performance and Section III describes UCT and UCT-C and compares UCT-C's asymptotic performance with UCB1's for deterministic MDPs.
In Section IV, we point out the limitations of UCT and UCT-C for stochastic MDPs in the perspective 
of the complexity and the asymptotic performance. We conclude the note in Section V.

\section{Setup and Problem Statement}

We consider a finite-horizon MDP with a finite state-set $X$ and a finite action-set $A$.
By taking an action $a$ in $A$ at a state $x$ in $X$, the state makes a transition
to $y$ in $X$ by the probability of $P_{xy}^a$ and a reward sample of $R(x,a)$ is
obtained, where $R(x,a)$ is a random variable associated with each pair of $(x,a)$
whose distribution is possibly unknown. 
We assume that the range of $R(x,a)$ is $[0,1]$ for any $x$ and $a$.

For $h\geq 1$, we define an $h$-horizon policy $\{\pi_t, t=0,...,h-1\}$ as a finite sequence of mappings of length $h$ where $\pi_t: X\rightarrow A$.
Let $\Pi_h$ be the set of all possible $h$-horizon policies.
Given $\pi \in \Pi_h$, 
define a random variable $X^{\pi}_t$ that denotes the state at time or level $t$ by following $\pi$,
where a random transition from $x$ at $t$ to $y$ at $t+1$ is made according to the probability of $P_{xy}^{\pi_t(x)}$.
When we \emph{roll out} or \emph{simulate} $\pi$ over $h$-transitions starting from $x$, it means 
that we follow $\pi$ over $h$-transitions given that $X_0^{\pi}=x$, creating a single sample-path.

Assume that a discounting factor $\gamma$ is fixed in $(0,1]$.
Let a random variable, $S^{\pi}_h, \pi\in \Pi_h$, be given
such that
\[
 S^{\pi}_h = \sum_{t=0}^{h-1} \gamma^{t} R \Bigl (X_t^{\pi},\pi_t(X_t^{\pi}) \Bigr ).
\] Define the \emph{value of rolling out} $\pi$ \emph{over} $h$-\emph{horizon} \emph{at} $x$ in $X$
as the conditional expectation $V^{\pi}_h(x) := E[S^{\pi}_h| X_0^{\pi}=x]$.

The problem is to find the \emph{optimal value at} $x$ in $X$ for a given horizon $H\geq 1$
defined as
\[ V^*_H(x) := \max_{\pi \in \Pi_H} V^{\pi}_H(x).\]
or to obtain an optimal policy $\pi^* \in \arg\max_{\pi \in \Pi_H} V^{\pi}_H(x)$ for all $x\in X.$ 
Throughout the note, we fix an initial state $x$ for the horizon $H$.

As a general approach to the MDP problem, we describe MCTS as the following algorithmic framework:
An algorithm (as an instance of MCTS) generates a sequence of the $H$-horizon policies $\{\rho^n, n\geq 1\}$, where $\rho^n\in \Pi_H$ is obtained from so-called ``\emph{index-function}" of the algorithm at $n$, and simulated over $H$-transitions starting at $x$ obtaining a sample of $S^{\rho^n}_H|X^{\rho^n}_0 = x$.
The algorithm outputs
\[
   \frac{1}{n} \sum_{k=1}^{n} S^{\rho^k}_H \Bigl | X_0^{\rho^k} = x
\] as an estimate of $V^*_H(x)$ at $n$.
The index-function at a given step is defined in general over $X\times A\times \{0,...,H-1\}$ and measures the utility of selecting (or sampling) an action at a state at a level.

The performance criterion of the algorithm is given by the expected absolute error of
\begin{equation}
   \Biggl | V^*_H(x) - E \Bigl [\frac{1}{n} \sum_{k=1}^{n} S^{\rho^k}_H \Bigr | X_0^{\rho^k} = x \Bigr ] \Biggr |
\end{equation} and the algorithm is referred to be \emph{asymptotically optimal} if the error goes to zero as $n\rightarrow \infty$.

\section{UCB1}

This section describes UCB1 as an instance of MCTS in our framework and provides its performance.
In the sequel, $[E]$ denotes the indicator function with the event $E$ inside Iverson brackets. If $E$ is true, $[E]=1$ and 0 otherwise. Let $\{\mu^n, n\geq 1\}$ be the sequence of the policies generated by UCB1.

Let $T^n_{\UCB}(\pi) = \sum_{k=1}^{n} [\mu^k=\pi]$ which denotes the number of times $\pi$ in $\Pi_H$ was chosen and simulated up to the step $n \geq 1$. For $\pi\in \Pi_H$ such that $T^n_{\UCB}(\pi)\neq 0$, $x\in X$, and $n\geq 1$, let
\begin{equation}
\label{ucbsampavg}
   Q^n_{\UCB}(\pi,x) = \frac{1}{T^n_{\UCB}(\pi)} \sum_{k=1}^{n} [\mu^k = \pi] S^{\pi}_{H} \Bigl | X_0^{\pi} = x.
\end{equation}  Thus, $Q^n_{\UCB}(\pi,x)$ is the sample average of $S^{\pi}_{H} \Bigl | X_0^{\pi} = x$ over the time steps at which $\pi$ was selected.
Note that the samples of $S^{\mu^n}_H | X_0^{\mu^n}=x$ and $S^{\mu^{n'}}_H | X_0^{\mu^{n'}}=x$ are independently generated and the distributions of $S^{\mu^k}_H | X_0^{\mu^k}=x$ and $S^{\mu^{k'}}_H | X_0^{\mu^{k'}}=x$ for any $k$ and $k'$ are \emph{same} if $\mu^k=\mu^{k'}$. It follows that
$E[Q^n_{\UCB}(\pi,x)] = E[Q^{n'}_{\UCB}(\pi,x)] = E[S^{\pi}_{H}| X_0^{\pi} = x]$ for any $n$ and $n'$.
Borrowing the term used in Shah \emph{et al.}~\cite{shah}, the MAB process that sequentially selects a policy (arm) and obtaining a random sample over time is \emph{stationary}.

For a given $n > 1$, let $\mathcal{I}^n_{\UCB}$ be a function over $\Pi_H\times X$ such that for 
$\pi\in \Pi_H$ and $x\in X$,
\begin{equation}
   \mathcal{I}^n_{\UCB}(\pi,x) = \begin{cases}
   Q^n_{\UCB}(\pi,x) + \sqrt{\frac{2\ln n}{T^n_{\UCB}(\pi)}} & \text{if } T^n_{\UCB}(\pi) \neq 0. \\
   \mathcal{I}_{\max} & \text{otherwise},
\end{cases}
\end{equation} where $\mathcal{I}_{\max}$ is set to be any constant bigger than $H + \sqrt{2\ln n}$. 

Due to the assumption that $\max_{y\in X, a\in A} R(y,a)\in [0,1]$, the maximum value that $\mathcal{I}^n_{\UCB}(\pi,x)$ can take is less than equal to $H + \sqrt{2\ln n}$ over all $x$ in $X$ and all $\pi\in \Pi_H$ such that $T^n_{\UCB}(\pi) \neq 0$.
The reason for introducing $\mathcal{I}_{\max}$ is that we enforce any of the policies that have \emph{not} been selected up to $n$ to have the equal priority of being chosen at $n+1$ but higher than any policy selected before.
In other words, this ensures that each policy in $\Pi_H$ is selected at least once.
We remark that $\mathcal{I}^n_{\UCB}(\pi,x)$ is actually the original index-function of UCB1 in~\cite{auer} defined in a slightly different form by introducing $\mathcal{I}_{\max}$ and this modification does not change the original functionality of UCB1. 
This definition just incorporates the condition that each arm is played at least once,
which the original description of the algorithm has, into the index-function $\mathcal{I}^n_{\UCB}$.

For the MAB problem~\cite{auer}, at each $n \geq 1$, UCB1 simply chooses a policy that achieves the maximum index-value $\max_{\pi\in \Pi_H} \mathcal{I}^n_{\UCB}(\pi,x)$ (with the ties broken arbitrarily) 
for playing the policy at $n+1$ with an arbitrarily chosen initial policy at $n=1$.
It is well known that UCB1 acts as a benchmark strategy for solving the MAB problems when the performance
is measured by ``the expected regret". In particular, UCB1 utilizes 
the upper confidence bound of the sample mean and the term related with the bound in the index-function 
plays an important role in \emph{exponentially} bounding the expected regret relative to the optimal value 
in probability~\cite{auer}.

To fit UCB1 into an algorithm in our framework, we simply adapt $\mathcal{I}^n$ into an index-function $I^n_{\UCB}$ defined over $X\times A\times \{0,1,...,H-1\}$ at $n > 1$. 
Note again that an initial state is fixed by $x$ for the horizon $H$. (That is, the underlying problem is obtaining $V^*_H(x)$.) $I^n_{\UCB}$ is simply given as follows: For $y\in X$, $a\in A$, and $l \in \{0,1,...,H-1\}$,
\begin{equation}
\label{indexUCB}
I^n_{\UCB}(y,a,l) = \begin{cases} 
\mathcal{I}^n_{\UCB}(\pi^n,x) & \text{if } a = \pi^n_l(y), \text{where } \pi^n \in \argmax_{\pi\in \Pi_H} \mathcal{I}^n_{\UCB}(\pi,x) \\
- \mathcal{I}^n_{\UCB}(\pi^n,x) & \text{otherwise},
\end{cases} 
\end{equation}
Then UCB1 starts with an arbitrary policy $\mu^1 \in \Pi_H$ and generates $\mu^{n}$ for $n > 1$ to roll out such that for $y\in X$ and $l\in \{0,1,...,H-1\}$,
\begin{equation}
\label{phiUCB}
  \mu^{n}_l(y)  \in \argmax_{a\in A} I^{n-1}_{\UCB}(y,a,l).
\end{equation} We can see that $\mu^{n} = \pi^n$ for all $n > 1$ as desired.

We write the result obtained by Auer \emph{et al.} as a theorem below in our terms.
\begin{thm}~\cite[Theorem 1]{auer}
\label{thm:ucb1}
Let $\{\mu^n, n\geq 1\}$ be the sequence of the policies in $\Pi_H$ generated by UCB1. 
For any $n\geq |\Pi_H|$ and $x$ in $X$,
\begin{equation}
\label{ucb1}
   0\leq V^*_H(x) - E\left [\frac{1}{n} \sum_{k=1}^{n} S^{\mu^k}_H \biggr | X_0^{\mu^k} = x \right] 
\leq O\Bigl ( \sum_{\pi: \Delta_{\pi}>0 } \frac{1}{\Delta_{\pi}} \frac{\ln n}{n} \Bigr ), 
\end{equation} where $\Delta_{\pi} := V^*_H(x) - V^{\pi}_H(x)$.
\end{thm}
The bound expressed with the big-$O$ notation can be further simplified by
\[
O \Bigl (\frac{|\Pi_H|}{\Delta_{\min}}\frac{\ln n}{n} \Bigr ) = O \Bigl (\frac{|A|^{|X|H}}{\Delta_{\min}}\frac{\ln n}{n} \Bigr ),
\]
where $\Delta_{\min} := \min_{\pi: V^{\pi}_H(x) < V^*_H(x)} \Delta_{\pi}$.

The upper bound holds regardless of the stochasticity of MDPs. No assumption
that an optimal policy is unique needs to be imposed. (On the contrary,
the unique existence assumption needs to be imposed to have the meaningful bound
presented by Shah \emph{et al.}~\cite{shah}.)
Significantly, Lai and Robbins~\cite{lai} showed that the upper bound of $O(\ln n/n)$ is 
``asymptotically optimal" or tight in that 
if the reward distribution associated with each arm satisfies some mild assumptions,
then
for \emph{any} algorithm that produces
a sequence of the arms to be played, the expected number of times any non-optimal 
arm $a$ has been played up to the step $n$ is lower bounded by 
$\ln n/n$ divided by the KL-distance between the reward distribution of $a$ and 
of an optimal arm if $n$ is sufficiently large. Therefore this result also applies to the
MCTS case. Any MCTS algorithm in our framework must achieve $O(\ln n/n)$-bound 
on the expected absolute error to be asymptotically optimal in the sense of
Lai and Robbin's result.

\section{UCT and UCT-C}
\subsection{UCT}

Let $\{\phi^n, n \geq 1\}$ be the sequence of policies generated by UCT. The initial state for the horizon $H$ is fixed with $x$.
Let the sample of the reward-sum obtained by rolling out $\phi^n$ over $(H-l)$-transitions starting from $y$ in $X$ at level $l$ in $\{0,1,...,H-1\}$ be $S^{\phi^n}_{H}(y,l)$ such that
\[S^{\phi^n}_{H}(y,l)  = \sum_{t=l}^{H-1} \gamma^{t} R \Bigl (X_t^{\phi^n},\phi^n_t(X_t^{\phi^n})\Bigr ) \Bigl | X_l^{\phi^n} = y.
\]
Let $T^n(y,l) = \sum_{k=1}^{n} [X_l^{\phi^k}=y]$, which denotes the number of times $y$ in $X$
was visited at level $l\in \{0,1,...,H-1\}$ up to step $n\geq 1$ by simulating the policies $\phi^1$,...,$\phi^n$.
For $y\in X$, $l \in \{0,1,...,H-1\}$, and $n \geq 1$, define
\begin{equation}
 V^n(y,l) = \begin{cases} 
 \frac{1}{T^n(y,l)} \sum_{k=1}^{n} S^{\phi^k}_{H}(y,l) [X_l^{\phi^k}=y] & \text{if } T^n(y,l) \neq 0 \\
 V_{\max} & \text{otherwise},
 \end{cases}
\end{equation} where $V_{\max}$ is set to be an arbitrary constant bigger than $H$
because the maximum value that $S^{\phi^k}_{H}(y,l)$ can take is less than equal to $H$.

Let also $T^n(y,a,l) = \sum_{k=1}^{n} [X_l^{\phi^k}=y,\phi^k_l(y)=a]$ that denotes the number of times $y$ was visited at level $l$ \emph{and} $a$ was taken at the visited $y$ up to the step $n$. For $y\in X$, $a\in A$, and $l \in \{0,1,...,H-1\}$ define
\begin{equation}
\label{QnUCT}
   Q^n(y,a,l) = \begin{cases}
   \frac{1}{T^n(y,a,l)} \sum_{k=1}^{n} S^{\phi^k}_{H}(y,l) [X_l^{\phi^k}=y, \phi^k_l(y)=a] & \text{if } T^n(y,a,l) \neq 0 \\
   Q_{\max} & \text{otherwise}
   \end{cases}
\end{equation} where $Q_{\max}$ is set to be an arbitrary constant bigger than $H$.
It can be easily seen that when $T^n(y,l)\neq 0$, we can rewrite
\[
       V^n(y,l) = \frac{1}{T^n(y,l)} \sum_{a\in A, T^n(y,a,l) \neq 0} T^n(y,a,l) Q^n(y,a,l).
\]

For $y\in X$, $a\in A$, and $l \in \{0,1,...,H-1\}$, define the index-function $I^n_{\UCT}$ of
UCT at $n \geq 1$ as
\begin{equation}
\label{indexUCT}
I^n_{\UCT}(y,a,l) =  \begin{cases}
     Q^n(y,a,l) + \sqrt{\frac{2\ln T^n(y,l)}{T^n(y,a,l)}} & \text{if } T^n(x,a,l) \neq 0 \\
     I_{\UCT}^{\max} & \text{otherwise},
\end{cases} 
\end{equation} 
where $I_{\UCT}^{\max}$ is set to be any constant bigger than $Q_{\max}+\sqrt{2\ln n}$.
Then UCT generates $\phi^n, n\geq 1$ to roll out as follows:
For $l \in \{0,1,...,H-1\}$ and $y\in X$,
\begin{equation}
  \label{phiUCT}
  \phi^{n}_l(y)  \in  \argmax_{a\in A} I^n_{\UCT}(y,a,l).
\end{equation} 

Observe first that in comparison with the case of UCB1, for any given $n$ and $n'$ with $n\neq n'$, $E[Q^n(y,a,l)] \neq E[Q^{n'}(y,a,l)]$ in general for any $y$ and $l$. 
This can be seen because the value of $Q^n(y,a,l)$ depends on the sequence $\{\phi^1,...,\phi^n\}$ but on the other hand, the value of $Q^{n'}(y,a,l)$ does on $\{\phi^1,...,\phi^{n'}\}$.
That is, the expected utilities of taking an action $a$ at $y$ at time $n$ and $n'$ are \emph{different}.
The MAB process induced by UCT, associated with $y$ at level $l$, is \emph{non-stationary} in that the reward distributions of the arms are time-varying. This necessarily affects the convergence behavior of UCT.
Indeed, it turns out that non-stationarity together with the logarithmic UCB-like term of the index-function of UCT makes 
it difficult to draw an exponential concentration in probability~\cite{shah}.
The very correction of Shah \emph{et al.}'s to UCT is thus to the UCB-like term 
to make the sequence of the expectations does converge but without preserving the 
exponential concentration in probability.

It should be also noted that the output $V^n(x,H)$ of UCT for the initial state $x$ is \emph{not} the usual process of Monte-Carlo simulation because the average value in $V^n$ (and $Q^n$) is computed over $n$ random sample-values of $n$ possibly \emph{different} random variables. The law of large numbers does not apply here.
Arguably, it is misleading to put the term ``Monte-Carlo" in front of tree search even though MCTS was coined with a different simulation method from UCT.
In a survey paper by Brown \emph{et al.}~\cite{browne}, for a general description of UCT-based MCTS, Monte-Carlo is referred to as the ``generality of random sampling".
Fixing a policy in finite-horizon MDP induces an inhomogeneous Markov chain and rolling out a policy is similar to a \emph{random walk over the chain}.
With this view, the Monte-Carlo part in MCTS probably corresponds to simulating policies.

Because of the average term of the samples from a non-stationary MAB process,
we see that the index-function of UCT given~(\ref{indexUCT}) follows only the \emph{form} of UCB1's.
The term $\sqrt{\frac{2\ln T^n(x,l)}{T^n(x,a,l)}}$ does not necessarily play the role of UCB.
Even if UCT stands for ``UCB applied to Trees," there exists a crucial difference between UCB in the index-function of UCT and UCB in UCB1.
This is another aspect of the difficulty of deriving an exponential bound in probability relative to the optimal value unlike the case in UCB1.

Furthermore, no relationship between $V^{n}(y,l)$ and $V^{n}(z,l')'s$ for $l'< l$ and $y,z\in X$
is explored while computing $V^n$. Even if $V^n(y,l)$ is an estimate
of $V^*_l(y)$, computation of $V^n$ \emph{does not} incorporate or approximate the optimal substructure property from the dynamic programming (DP) equations.
In fact, there seems not to exist any ``optimality" substructure between the \emph{non-stationary} MAB process associated with a state $y$ in a level $h-1$ and the \emph{non-stationary} MABs associated with the visitable states $z$'s from $y$ in the level $h$.
It is the convergence behaviour of UCT that is expected to in a way \emph{search for such optimality relations} between the levels as the non-stationary MABs \emph{approach to} stationary MABs.
Indeed, Shah \emph{et al.}'s convergence proof for their UCT-C that does not explore DP relation \emph{is based on} emulating the DP-algorithm of backward induction (or finite-horizon value-iteration).

\subsection{UCT-C}
\label{sec:UCT-C}

Shah \emph{et al.} found a fundamental limitation in the usage of the UCB1-like selection in the index-function for achieving a logarithmic convergence-rate from non-stationarity of the MABs induced in UCT.
They resolved this by modifying the term $\sqrt{\frac{2\ln T^n(x,l)}{T^n(x,a,l)}}$ in the index-function of UCT with a certain \emph{polynomial expression} (called ``polynomial bonus term" in~\cite{shah}) of $T^n(x,l)$ and $T^n(x,a,l)$
and with some (algorithmic) constants to be set for each level that all together ``control" the concentration behaviour of the resulting algorithm in probability. 

Assume that a \emph{deterministic} MDP is given such that any transition from a state is deterministic. That is, for any $\pi\in \Pi_H$, there exists $z\in X$ such that $P_{yz}^{\pi_l(y)}=1$
for any $l\in \{0,...,H-1\}$ and for any $y\in X$.
(Shah \emph{et al.} presented their main results with rigorous analysis only under the deterministic MDP setting and then discussed an idea of extending UCT-C for stochastic MDPs in the appendix.
To compare the convergence rate of UCT-C with UCB1's, we follow the method of their exposition.)

By the modification of UCT into UCT-C, each non-stationary MAB process that occurs during the invocation of UCT-C satisfies the certain properties of
the convergence and the concentration associated with the states visited from the initial root state $x$ (see, Section 5 and 7~\cite{shah}).
The key idea of their analysis is to bound recursively the errors in the levels by emulating the backward 
induction based on Bellman's optimality equation.
Starting with bounding $V^*_0$, $V^*_h$ is bounded from $V^*_{h-1}$ inductively for $h>1$ by
incorporating the convergence results of the associated non-stationary MABs, providing a final bound on $V^*_H$.

Theorem below provides Shah \emph{et al.}'s result with the parts only relevant to our discussion 
with a simplification and a slight modification in our terms.
\begin{thm}~\cite[Theorem 1]{shah}
\label{thm:uct-c}
Assume that a deterministic finite-horizon MDP with an initial state $x\in X$ is given and that an optimal policy that achieves $V^*_H(x)$ is unique. Under some conditions on the parameters of UCT-C, $\{\pi^n, n\geq 1\}$ generated by UCT-C satisfies that
\begin{equation}
\label{boundhidden}
\Bigl |V^*_H(x) - E\Bigl [\frac{1}{n} \sum_{k=1}^{n} S^{\pi^k}_H \Bigr | X_0^{\pi^k} = x \Bigr ] \Bigr | \leq
O\left (\frac{H|A|}{\min_{h}(\Delta_{\min}^{h})^2}\frac{1}{\sqrt{n}} \right ),
\end{equation} where for $h \in \{1,...,H-1\}$, $\Delta_{\min}^{h} = \min_{\pi^h\in \Pi_h}\Delta_{\pi^h}$ with $\Delta_{\pi^h} := \min_{y\in X} (V^*_h(y) - V^{\pi^h}_h(y))$. 
\end{thm}
\vspace{0.5cm}

It should be noted again that in contrast to the above result, the result for UCB1 in Theorem~(\ref{thm:ucb1}) holds 
for stochastic MDPs without any uniqueness assumption.
Lemma 5 and 6 in~\cite{shah} do not explicitly state the assumption that
the maximizer in the set of argmax is unique. It is imposed before the statement of
each lemma, respectively.
While the assumption of the uniqueness of an optimal arm has been usually made
in the best-arm identification for the MAB problems,
(see, e.g.,~\cite{bubeck}), putting a uniqueness assumption of an optimal
policy on the MDP model is \emph{not general}.
In fact, there exists conditions under which MDPs have a unique optimal policy (see, e.g.~\cite{cruz}).

Intuitively, in order for the average value of the samples to converge to the optimal value in a non-stationary MAB, an optimal arm should be played as often as possible and ``sufficiently more often" than non-optimal arms so that the samples from the optimal arm contributes to the average sufficiently and the samples from the non-optimal arms become negligible. 
Shah \emph{et al.}'s result basically means that as $n$ increases, the optimal arm gets \emph{identified} (in probability) by playing the arms according to the index-function of UCT-C (where the analysis of this part necessitates the uniqueness assumption) and the optimal arm is played polynomially more often than non-optimal arms (in contrast with the UCB1 case of ``exponentially more often") 
and eventually the average converges to the optimal value with the rate 
which coincides with the case of the law of large numbers, i.e., $O(1/\sqrt{n})$.
UCT-C is asymptotically optimal.

In the statement of Theorem 1 in~\cite{shah}, \emph{no constant factors are written
inside of the $O$-notation}. Lemma 6 in~\cite{shah} provides 
the parameter $\delta_n^{(h-1)}$
in terms of several factors, 
including $|A|$, a complexity parameter $\Delta_{\min}^{(h-1)}$ determined from the visited
states at level $h-1$ of the tree UCT-C has built, and some algorithmic constants of UCT-C
at level $h-1$ and $h$.
Even if some explicit recurrence relations among parameters between level $h$ and $h-1$, 
e.g., $\alpha^{(h)}, \eta^{(h)}$ are given, \emph{no recursive relation} between $\delta_n^{(h-1)}$ 
and $\delta_n^{(h)}$ is explicitly given.
If so in the paper, an inductive argument on $h$ would provide 
a bound on $\delta_n^{(0)}$. The bound on $\delta_n^{(0)}$ would play a critical role in bounding $V^*_H(x)$.

In order to
find a meaningful hidden constant-factor in $O(1/\sqrt{n})$, the result of 
Lemma 5 in~\cite{shah} 
can be re-examined while applying the inductive reasoning of Shah \emph{et al.}'s.
Considering only \emph{the errors contributed by the MDP model-parameters}
leads to having $H|A| / \min_{h}(\Delta_{\min}^{h})^2$ factor in the bound in our statement.

Now then from the result of UCB1 in Theorem~\ref{thm:ucb1},
we can state that UCB1 is asymptotically faster than UCT-C as a theorem below.
\begin{thm}
Suppose that the assumption in Theorem~\ref{thm:uct-c} holds. 
If $\{\mu^n\}$ and $\{\pi^n\}$ are the sequences of the policies generated by UCB1 and UCT-C, respectively, then the sequence of
$\bigl \{ | V^*_H(x) - E[n^{-1} \sum_{k=1}^{n} S^{\mu^k}_H \Bigr | X_0^{\mu^k} = x] | \bigr \}$
 converges to zero faster than 
$\bigl \{ | V^*_H(x) - E[n^{-1} \sum_{k=1}^{n} S^{\pi^k}_H \Bigr | X_0^{\pi^k} = x] | \bigr \}$
as $n\rightarrow \infty$.
\end{thm}
\begin{proof}
The proof is trivial. From the assumption, at every state maximally $|A|$ different states can 
be reached.
Therefore, the upper bound of UCB1 in~(\ref{ucb1}) becomes tighter by replacing $|A|$ by $\min\{|A|,|X|\}$ having the factor $(\min\{|A|,|X|\})^H$ instead of $|A|^{|X|H}$.
Furthermore, both $\Delta_{\min}$ and $\min_{h} (\Delta_{\min}^h)^2$ is $\Theta(1)$. Therefore, comparing the bound $O(\ln n /n)$ of UCB1 and $O(\sqrt{n}/n)$ of UCT-C,
for any given value of $|X|$, $|A|$, $|H|$, $\Delta_{\min}$, and $\Delta_{\min}^h, h\in \{1,...,H-1\}$, there exists a corresponding sufficiently large $k < \infty$ such that the absolute error made by UCB1 is smaller than that by UCT-C for all $n\geq k$.
\end{proof}

In Fig.~\ref{fig1}, we show a typical convergence behavior of the absolute error difference between UCB1 and UCT-C for a simple deterministic MDP with $|X|=10, |A|=2, H=15,$ and $\Delta_{\min} = \min_{h} (\Delta_{\min}^h)^2 = 0.1$.
When $n$ is relatively small, the effect of the large (exponential in $H$) constant factor of UCB1 compared with (polynomial in $H$) that of UCT-C is apparent.
As $n$ increases, the effect becomes gradually negligible and the difference becomes smaller and approaches to a cut point crossing the horizontal axis reaching the negative area.
\begin{figure*}[!tb] \centering
\includegraphics[height=10cm]{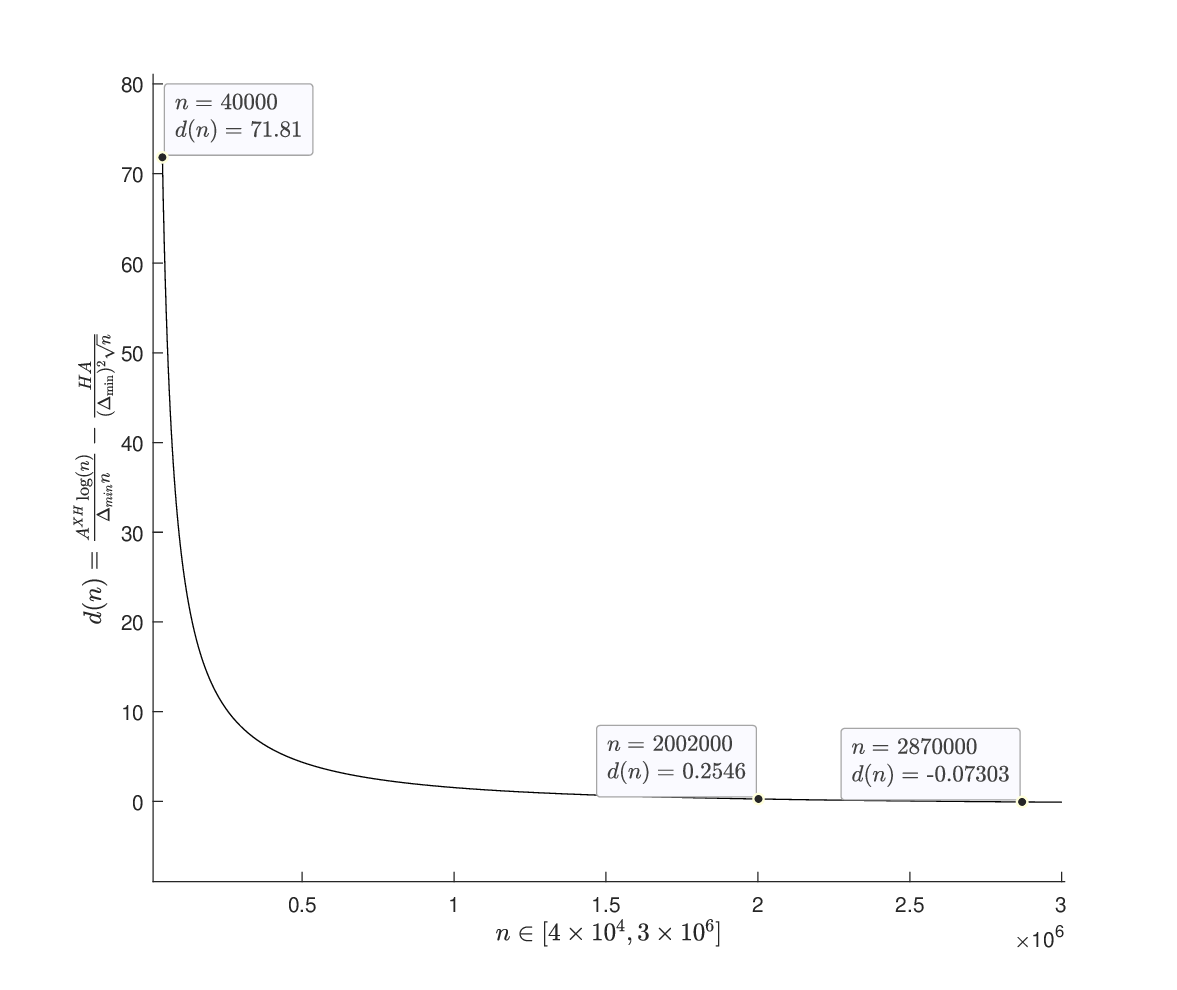}
\caption{Convergence behavior comparison of UCB1 and UCT-C error-estimates by the upper bound difference}
\label{fig1}
\end{figure*}

At this point, it should be clarified that we are \emph{not} claiming that UCB1 is a better choice than UCT-based MCTS 
when we actually consider implementing MCTS, in particular, with possible adjustments to the domains 
or adaptations into the problem characteristics.
Our goal is to point out the absence of theoretical supports of the empirical successes of UCT-based MCTS 
by comparing UCT-C with UCB1. More theoretical studies about UCT-based MCTS need to be done.

\section{The case of stochastic MDPs}

MCTS is supposed to overcome the curse of dimensionality problem in MDPs.
UCT's main design goal was this.
However, we have a substantial algorithmic problem that both UCT and UCT-C (or UCT-based MCTS in general)
need 
to update the value of the index-function \emph{whenever a state is visited} by rolling out
a policy.
Every visitable or reachable state from the initial (root) state needs to be added into
a tree-structure that each algorithm maintains, if visited for the first time, and whenever the state is
visited again, $Q^n$ and $V^n$ values need to be updated accordingly.
This makes the computational complexity depends on $|X|$.

A methodology of how to deal with stochastic MDPs is given for UCT-C 
by Shah \emph{et al.} in the appendix in~\cite{shah}. The key is to \emph{reduce} a stochastic MDP into an ``equivalent"
deterministic MDP and then to apply
the reasoning used for the deterministic case into the stochastic MDP for the
performance analysis.
(No complete analysis is provided there but it suffices for our message.)
The reduction idea is that
$Q^n(y,a,l)$ in~(\ref{QnUCT}) is changed into a \emph{weighted sum} of 
$V^n(z,l+1)$ with \emph{each possible next-state} $z$, with
weights being the empirical frequency of visiting each next-state thus far
and a polynomial bonus term with the same form as in~(\ref{phiUCT})
for each action is incorporated but with different algorithmic constants.
Shat \emph{et al.} showed that with this change, similar 
convergence and polynomial concentration properties (in probability) hold for each
non-stationary MAB associated with \emph{each state visited}.
By viewing then the children nodes associated with one action collectively as one 
``meta-node" corresponding to the action and applying the inductive reasoning used
for the deterministic MDP to the reduced MDP,
Shah \emph{et al.} claim that this leads to a convergence-rate of $O(1/\sqrt{n})$
for stochastic MDPs.

Most notably, the resulting UCT-C has a time-complexity that depends on $|X|$
besides $O(|X||A|H)$ space-complexity.
Because at the worst case, every state in $X$ can be visited at a state 
by taking an action (as long as the transition probability is positive), 
the time-complexity of updating $Q^n$ and $V^n$ with the weighted sum depends on $|X|$.
If a tree data-structure is used with the resulting UCT-C, the tree has 
$O((|A||X|)^H)$-size at the worst case.
In the deterministic setting, $|X|$-factor was \emph{one} so that 
the dependence on $|X|$ could be ignored.

The next issue is more critical. It is again related with hidden constant 
factors in the $O$-notation.
In the meta-node, each visited state from its parent node is 
associated with a nonstationary MAB whose contribution in $Q^n$ is estimated
by the frequency of visiting thus far so that the parent node needs to be
\emph{sufficiently visited often} to estimate the probability of transition.
However, in order to achieve a same degree of the error for each non-stationary 
MAB associated with each child, each child needs to be visited sufficiently often. 
For example, 
suppose that the true transition-probability to a child node is very small by taking an
action but has a very large (or compensable) optimal value at the child node.
In order for UCT-C to figure out this, the child node must be visited
sufficiently often from its parent node by taking that action and the parent node 
needs to visited sufficiently often.
This means that the value of the step $n$ to achieve the relative distance to $V^*_H(x)$
by $\epsilon>0$ in the
deterministic setting and that of $n$ to achieve the same error bound
$\epsilon$ in the stochastic setting must be \emph{very different} even if the rates of
$O(1/\sqrt{n})$ expressed with the big-$O$ are the same.
In other words, UCT-C achieves the desired error bound of $\epsilon$
in the stochastic setting at a ``sufficiently larger" step than in the deterministic setting.

Even if these side effects from the reduction are not rigorously explained in the paper,
it is obvious that the convergence rate of UCT-C in the stochastic setting should
include a non-negligible constant factor that reflects the transition
structure of MDP. For example, if there exists $\beta >0$ such that
$\inf\{P_{yz}^a | P_{yz}^a\neq 0, y,z\in X, a\in A\} \geq \beta$, the number of the next-states 
reachable from any state by taking any action is bounded by 
$\min\{|X|,\lfloor \beta^{-1} \rfloor\}$ (see, Appendix A~\cite{shah}). In this case,
the error bound would have the form of
\[
O\Bigl ( \min\{|X|,\lfloor \beta^{-1} \rfloor\} \cdot \frac{H|A|}{\min_{h} (\Delta_{\min}^h)^2} \frac{1}{\sqrt{n}} \Bigr ).
\]

It would be very likely that the convergence rate of UCT-C in the stochastic setting 
is much slower than in the deterministic setting and it depends on the state set size
in the worst case.
Because the bound of UCB1 in~(\ref{ucb1}) holds \emph{even for} stochastic MDPs,
UCB1's performance would become more competitive to that of UCT-C in terms of the convergence rate.

\section{Concluding Remarks}

The study in this note brings up a fundamental open-question whether it is possible to 
characterize in what conditions UCT-based MCTS works well for stochastic MDPs in general.
There is no theoretical back-up yet that can explain the empirical successes of UCT-based MCTS.

For example, one can consider the case where the size of the set of the
visitable states (with high probabilities) from the root initial state is 
relatively small and the states with low-probability reachabilities
have negligible optimal values. In other words, UCT-based MCTS need to 
\emph{control} somehow the number of the visitable states, or if possible, which 
state to visit.
It can then be speculated that because UCT or UCT-C can 
focus on visiting highly probable next-states, it might be effective.
Another point is 
that these algorithms do not incorporate Bellman's optimality principle
while updating the estimates. Combining this into the algorithms somehow can merit 
further investigation. A rigorous theoretical development is challenging.

However, non-stationarity does not disappear in this case too.
As long as the average of the samples from non-stationary MABs is used as the estimate
of the optimal value,
achieving an exponential concentration bound appears difficult as noted in~\cite{shah}.
This leads to a future (theoretical) topic of developing a variant of MCTS that 
can provide a faster convergence rate while still using the non-stationary MABs.
As we remarked before, the best asymptotic convergence-rate achievable
by an instance algorithm in our framework is $\Theta(\ln n/n)$.

Even if not formally stated anywhere, it can be argued that UCT is asymptotically optimal.
The UCB-like term in the index-function of UCT, as in the UCB term of UCB1, also
controls the frequency of playing each arm. Each arm is played infinitely often
due to the bonus term.
This property implies that while running UCT for a stochastic MDP, any visitable state 
from the root state is visited infinitely often and each arm is played infinitely 
often at every visited state. A similar inductive reasoning to UCT-C can be applied
in an asymptotic sense.
The main point here is about the convergence rate, not the convergence.

Finally, an experimental investigation that compares the performances of the algorithms in our framework to real applications would be interesting and support the theoretical comparative results.

\end{document}